\documentclass[11pt]{amsart}
\usepackage[colorlinks=false]{hyperref}
\usepackage{multicol}
\usepackage{enumerate}
\usepackage{mathtools}
\usepackage{amsfonts}
\usepackage{amssymb}
\usepackage{xypic}
\usepackage{mathrsfs}
\usepackage{graphicx}
\usepackage{url}
\usepackage{microtype}
\usepackage[margin=1in]{geometry}
\usepackage{color}
\usepackage{comment}
\usepackage{epsfig}
\usepackage{amsopn}
\usepackage{xypic}
\usepackage{mathrsfs}
\usepackage{array}
\usepackage[usenames,dvipsnames]{pstricks}
\usepackage{epsfig}
\usepackage{pst-grad} 
\usepackage{pst-plot} 
\usepackage[space]{grffile} 
\usepackage{etoolbox} 
\makeatletter 
\patchcmd\Gread@eps{\@inputcheck#1 }{\@inputcheck"#1"\relax}{}{}
\makeatother

\newtheorem{thm}{Theorem}[section]
\newtheorem{prop}[thm]{Proposition}

\newtheorem{cor}[thm]{Corollary}

\newtheorem{lemma}[thm]{Lemma}
\newtheorem*{thm*}{Theorem}						
\newtheorem*{prop*}{Proposition}
\newtheorem*{lemma*}{Lemma}
\newtheorem*{cor*}{Corollary}
\newtheorem*{conj*}{Conjecture}

\theoremstyle{definition}
\newtheorem{definition}[thm]{Definition}

\theoremstyle{remark}

\renewcommand{\P}{\mathbb{P}}

\let\hom\relax 
\DeclareMathOperator{\hom}{hom}
\DeclareMathOperator{\Hom}{Hom}

\DeclareMathOperator{\Ext}{Ext}
\DeclareMathOperator{\ext}{ext}

\DeclareMathOperator{\pic}{Pic}
\let\O\relax
\DeclareMathOperator{\O}{\mathcal{O}}

\DeclareMathOperator{\rk}{rk}

\DeclareMathOperator{\gr}{gr}

\usepackage{color}
\usepackage[dvipsnames]{xcolor}
\makeatletter
\let\c@theorem\c@figure
\makeatother
\begin{document}

\title[Non--globally generated bundles on curves]{Non--globally generated bundles on curves}

\author[J. Kopper]{John Kopper}
\address{Department of Mathematics, The Pennsylvania State University, University Park, PA, 16802}
\email{kopper@psu.edu}

\author[S. Mandal]{Sayanta Mandal}
\address{Chennai Mathematical Institute, H1, SIPCOT IT Park, Siruseri Kelambakkam 603103, India}
\email{smanda9@uic.edu}

\title{Non--globally generated bundles on curves}

\thanks{During the preparation of this article, both authors were partially supported by NSF RTG grant DMS-1246844.}

\begin{abstract}
We describe the locus of stable bundles on a smooth genus $g$ curve that fail to be globally generated. For each rank $r$ and degree $d$ with $rg<d<r(2g-1)$, we exhibit a component of the expected dimension. We show moreover that no component has larger dimension and give an explicit description of those families of smaller dimension than expected. For large enough degrees, we show that the locus is irreducible.
\end{abstract}

\maketitle
\setcounter{tocdepth}{1}
\tableofcontents

\section{Introduction}
The central goal of classical Brill-Noether theory on curves and its higher-rank analogues is to describe loci of vector bundles possessing unexpectedly many global sections. In this paper we complement this study by describing the locus of stable vector bundles that fail to be globally generated. In the case of line bundles, the picture is quite clear (\S\ref{sec:linebundles}):

\begin{prop}
Let $C$ be a smooth projective curve of genus $g \geq 2$ and $d$ an integer. Then the following hold.
\begin{enumerate}[(a)]
    \item If $d \leq g$, then the general line bundle of degree $d$ is not globally generated.
    \item If $g+1 \leq d \leq 2g-1$, then the general line bundle of degree $d$ is globally generated and the locus of line bundles with basepoint has codimension $d-g$ in $\pic^d(C)$.
    \item If $d \geq 2g$, then every line bundle of degree $d$ is globally generated.
\end{enumerate}
\end{prop}

The locus of non--globally generated line bundles can often be described explicitly. For example, when $d=2g-1$, the locus is isomorphic to the curve $C$. This follows from the fact that every degree $2g-1$ line bundle with a basepoint $p \in C$ is of the form $\O(K+p)$, where $K$ is a canonical divisor.

The problem for higher-rank bundles is more subtle. While statements directly analogous to (a) and (c) in the above proposition still hold (Prop. \ref{prop:easy_vb_facts}), it is harder to understand the locus of non--globally generated bundles when the general bundle is globally generated. For example, given a non--globally generated line bundle $L$, vector bundles of the form $E = F \oplus L$ are never globally generated, and one can produce families of such bundles in arbitrary degree. Another difficulty is that global generation is not well-defined for S-equivalence classes (Prop. \ref{propn:global_generation_S_equiv}). Our approach is to study the locus $N_C(r,d) \subset U_C(r,d)$ of \emph{stable} bundles of rank $r$ and degree $d$ that fail to be globally generated. When $rg+1 \leq d \leq r(2g-1)-1$, this locus is non-empty (Prop. \ref{thm:existence_stable_seq}), and has codimension at least one (Prop. \ref{prop:easy_vb_facts}).

Using a theorem of Sundaram, we can produce an upper bound on $\dim N_C(r,d)$. A stable $E \in U_C(r,d)$ has a basepoint at $p$ if and only if $h^1(E(-p))>0$, or equivalently, $h^0(E^\ast(K+p)) > 0$. If $rg+1 \leq d \leq r(2g-1)-1$, then the vector bundle $E^\ast(K+p)$ has degree $d'=r(2g-1)-d$ satisfying
\[
1 \leq d' \leq r(g-1)-1.
\]
Thus the Brill-Noether locus
\[
W^0_{r,d'} = \{F \in U_C(r,d'): h^0(F) \geq 1\}
\]
surjects onto the locus of $E \in U_C(r,d)$ with basepoint at $p$ via the map $U_C(r,d') \to U_C(r,d)$ defined by $F \mapsto F^\ast(K+p)$. 

A theorem of Sundaram \cite[Thm II.3.1]{sundaram} says that $W^0_{r,d'}$ has a unique component of maximal dimension
\[
r^2(g-1)+1 - (rg-r-d'+1) = \dim U_C(r,d) - (d-rg+1).
\] 
By varying $p$ along $C$, we conclude that the dimension of $N_C(r,d)$ is bounded above by
\begin{equation}\label{eqn:expected_dimension}
    r^2(g-1)+1 - (d-rg).
\end{equation}

We call (\ref{eqn:expected_dimension}) the \emph{expected dimension} of $N_C(r,d)$, and $d-rg$ its \emph{expected codimension}. Our goal is to study the dimension and irreducible components of $N_C(r,d)$. The main results are summarized in the following theorem.

\begin{thm}[\protect{\ref{prop:determinantal_description}, \ref{cor:family_of_exp_dim}, \ref{thm:irreducibility_n0}}]
Let $C$ be a smooth projective curve of genus $g\geq 2$, and $r$ and $d$ integers satisfying $r \geq 2$ and $rg+1 \leq d \leq r(2g-1)-1$. Then we have the following:
\begin{enumerate}[(a)]
    \item $N_C(r,d)$ is nonempty and has a component of the expected dimension and no component of larger dimension.
    \item Let $N_C^0(r,d) \subset N_C(r,d)$ denote the set of stable, non--globally generated $E$ with $h^1(E) = 0$. Then $N_C^0(r,d)$ is nonempty, has a component of the expected dimension, and no component of codimension greater than $d-rg+1$.
    \item Suppose $rg+g-1 \leq d \leq r(2g-1)-1$. Then $N_C(r,d)$ is irreducible of the expected codimension $d-rg$.
\end{enumerate}
\end{thm}

The key ingredients in the proof of the above theorem are the results of Teixidor i Bigas and Russo \cite{teixidor-russo} and Narasimhan and Ramanan \cite{narasimhan-ramanan}. We also investigate the locus of stable bundles $E$ with finitely many basepoints and $h^1(E)=0$ in Proposition \ref{lemma:dense_openness_nc0_ncf}. This locus is comparatively well-behaved and has technical properties that simplify the proofs of some of the above results.

\subsection*{Structure of this paper} In Section \ref{sec:preliminaries} we briefly recall some basic facts about stable vector bundles on curves. Section \ref{sec:linebundles} is devoted to explaining the structure of the space of non--globally generated line bundles. In Section \ref{sec:bundles} we prove our main results concerning the locus $N_C(r,d)$ for bundles of rank at least $2$.

\subsection*{Acknowledgments} The authors would like to thank Montserrat Teixidor i Bigas, Izzet Coskun, Jack Huizenga, and John Lesieutre for their feedback and many valuable conversations. We are also grateful to the anonymous referees for their thoughtful feedback.

\section{Preliminaries}\label{sec:preliminaries}
In this section we collect some useful definitions and basic results about coherent sheaves on algebraic curves. We refer to \cite{huybrechts-lehn} for more background on stability and moduli spaces. We refer to the appendix in \cite{hoffman} for direct proofs of the necessary facts about stacks.

Let $C$ be a smooth projective curve of genus $g \geq 2$. We will denote by $K$ a fixed but arbitrary canonical divisor. Let $E$ be a torsion-free sheaf on $C$. We denote the \emph{slope} of $E$ by the number
\[
\mu(E) = \frac{\deg(E)}{\rk(E)}.
\]
We say $E$ is \emph{semistable} if $\mu(F) \leq \mu(E)$ for all proper subsheaves $F \subset E$, and we say $E$ is \emph{strictly semistable} if $\mu(F) = \mu(E)$ for some $F$. We say $E$ is \emph{stable} if strict inequality always holds.

Every semistable vector bundle $E$ admits a \emph{Jordan-H\"older} filtration
\[
0 = E_0 \subset E_1 \subset \cdots \subset E_n = E,
\]
where the $E_i$ are vector bundles and the quotients $E_i/E_{i+1}$ are stable. We write $\gr E = \bigoplus_i E_i/E_{i+1}$, and we say $E$ and $F$ are \emph{S-equivalent} if $\gr E \cong \gr F$.

We denote by $U_C(r,d)$ the moduli space parameterizing S-equivalence classes of semistable sheaves on $C$ with rank $r$ and degree $d$. The space $U_C(r,d)$ is irreducible of dimension $r^2(g-1)+1$.

A sheaf $E$ is called \emph{globally generated} if for all points $p \in C$, the natural map
\[
H^0(E) \to E_p
\]
is surjective. If $p$ is a point for which the above map is not surjective, we call $p$ a \emph{basepoint} of $E$. The following lemma is well-known and will be used implicitly throughout the paper.

\begin{lemma}
Suppose $E \in U_C(r,d)$. Then $p$ is a basepoint of $E$ if and only if $h^1(E(-p)) > h^1( E).$
\end{lemma}

We note that it is important to consider those bundles that are stable and discard those that are strictly semistable. Indeed, global generation is not always well-defined for S-equivalence classes as demonstrated by the following proposition.

\begin{prop}\label{propn:global_generation_S_equiv} 
Let $r = 2$, $d = 2(2g-1)$. Let $p$ be a point in $C$, $L$ a globally generated line bundle of degree $2g-1$, and $E$ a nontrivial extension of $L$ by $\O (K+p)$. Then $E$ is globally generated, semistable, and S-equivalent to a non-globally generated bundle.
\end{prop}

\begin{proof} The bundle $E$ fits into an exact sequence
\begin{equation}\label{eqn:sequivalence_gg}
   0 \to \O(K+p) \to E \to L \to 0 
\end{equation}
Let $F \subset E$ be a subsheaf. Then there is a map $F \to L$ whose image either has degree at most $2g-1$ (because $L$ is irreducible) or is zero, in which case $F \to L$ factors through $\O(K+p)$ and the analogous fact applies. Thus $E$ is semistable. The bundles $E$ and $\O(K+p) \oplus L$ are visibly S-equivalent and $\O(K + p) \oplus L$ is not globally generated. 

To show $E$ is globally generated we first note that if $E$ has a basepoint then it is at $p$. The end of the long exact sequence in cohomology applied to (\ref{eqn:sequivalence_gg}) after twisting by $\O(-p)$ is the following:
\[
\cdots \to H^0( L(-p)) \to H^1( \O(K)) \to H^1( E(-p)) \to 0.
\]
Since $h^1( \O(K)) = 1$, we see that $E$ will be globally generated if and only if the map $H^0(L(-p)) \to H^1( \O(K))$ is nonzero. Applying $\Hom(L, -)$ to (\ref{eqn:sequivalence_gg}) instead, we have:
\[
0 \to \Hom(L, \O(K+p)) \to \Hom(L, E) \to \Hom(L,L) \to \Ext^1(L, \O(K+p)) \to \cdots
\]
Note that $\Ext^1(L, \O(K+p)) = H^1( L^\ast(K+p))$. The image of $1_L$ under the map $\Hom(L,L) \to H^1( L^\ast(K+p))$ determines the map $H^0( L(-p)) \to H^1( \O(K))$ via the pairing
\[
H^1( L^\ast(K+p)) \times H^0(L(-p)) \to H^1( \O(K))
\]
from Serre duality. The image of $1_L$ in $\Ext^1(L, \O(K+p))$ is nonzero if (\ref{eqn:sequivalence_gg}) is assumed to be a nontrivial extension. Thus the map $H^0(L(-p)) \to H^1( \O(K))$ is nonzero and therefore $E$ is globally generated.
\end{proof}
Let $r\geq 2$ and $d$ be integers. We define
\begin{definition}
\[
N_C(r,d) = \{E : E \text{ is stable and not globally generated} \} \subset U_C(r,d).
\]
\end{definition}

\section{Line bundles with basepoints}\label{sec:linebundles}
We describe here the locus of line bundles that fail to be globally generated. The main idea is that the dimension of this locus decreases with the degree. If $d \leq g$, then the general line bundle of degree $d$ is not globally generated, and if $d > 2g-1$, then every line bundle is globally generated. Consequently, we are interested in the range $g+1 \leq d \leq 2g-1$. We will see that when $d=g+1$, the locus of line bundles in $\pic^d(C)$ with basepoint is a divisor. When $d=g+2$, it has codimension 2, and so on. The method of proof is to give an explicit description of line bundles that fail to be globally generated.

\begin{lemma}\label{lemma:linebundles_with_bp}
Let $D$ be a divisor of degree $d=g+j$ for $1 \leq j \leq g-1$. If $\O(D)$ has a basepoint at $p \in C$, then $D$ is linearly equivalent to $K+p - q_1 - \cdots -q_{g-j-1}$ for some $\{q_i\} \subset C$ satisfying $q_i \neq p$ for all $i$.
\end{lemma}
\begin{proof}
Since $\O(D)$ has a basepoint at $p$, we have $h^1(\O(D-p)) = h^1(\O(D)) + 1$. Equivalently, the map
\[
\Hom(\O(D-K-p),\O(-p)) \to \Hom(\O(D-K-p), \O)
\]
has a 1-dimensional cokernel. Thus there is a nonzero morphism $\O(D-K-p) \to \O$ whose image is not contained in the ideal sheaf $\O(-p)$. Thus $\O(D-K-p)$ is isomorphic to an ideal sheaf of a divisor $D' \subset C$ whose support does not contain $p$. Writing $D' = \sum q_i$ gives the result.
\end{proof}

The lemma shows that every line bundle $L$ with a basepoint at $p$ is of the form
\[
L \cong \O(K+p - q_1 - q_2 - \cdots - q_{g-j-1}).
\]
By varying the points $p$, $q_1, \dots, q_{g-j-1}$, we obtain a $(g-j)$-dimensional family of line bundles with basepoint. \emph{Prima facie}, it may occur that many of these bundles are isomorphic and that the locus cut out in $\pic^d(C)$ does not have dimension $g-j$. The next proposition shows that the generic situation is that a given line bundle in this family is isomorphic to only finitely many others in the family, and therefore that the locus does have the expected dimension.

\begin{prop}
Suppose $1 \leq j \leq g-1$. Then the dimension of the locus of line bundles of degree $d=g+j$ having a basepoint equals $g-j$.
\end{prop}
\begin{proof}
Let
\[
V =\{(q_1 + \cdots + q_{g-j-1}, p) : p \neq q_i \text{ for all } i\}\subset C^{(g-j-1)} \times C.
\]
Define a map $\Phi:V \to \pic^{g+j}(C)$ by
\[
\Phi(q_1+\cdots+q_{g-j-1},p) = \O(K+p-q_1 - \cdots - q_{g-j-1}).
\]

Lemma \ref{lemma:linebundles_with_bp} shows that $\Phi$ surjects onto the locus of non-globally generated line bundles. We will show that $\Phi$ is generically finite and it will follow that the dimension of the locus of degree $d$ line bundles with a basepoint is bounded above by $\dim V = g-j$. To conclude the proof, we will exhibit a $(g-j)$-dimensional family of line bundles with basepoint.

Given $(q_1 + \cdots + q_{g-j-1},p) \in V$ and $p' \in C$, we claim that $h^0(q_1 + q_2 + \cdots + q_{g-j-1} + p' -p) > 0$ is equivalent to the existence of points $q'_1, \cdots , q'_{g-j-1}$ in $C$ such that $(q'_1 + \cdots + q'_{g-j-1}, p') \in V$ and the corresponding line bundles $\Phi(q_1+ \cdots + q_{g-j-1},p)$ and $\Phi(q'_1+ \cdots + q'_{g-j-1}, p')$ are isomorphic. 

 Indeed, $K+p-\sum q_i$ is linearly equivalent to $K+p'-\sum q_i'$ if and only if
\[
q_1 + \cdots + q_{g-j-1} + p' - p \sim q_1' + \cdots + q_{g-j-1}',
\]
and the right-hand side is an effective divisor. If the $q_i$'s and $p'$ are general, then we know from \cite[Lemma IV.1.7]{acgh} (with $d = g-j$ and $r = 0$ in the notation in \emph{loc. cit.}) that $h^0(q_1 + \cdots + q_{g - j -1} + p') = 1$. Moreover, if  $p$ is not a basepoint of the (effective) divisor $q_1 + \cdots + q_{g-j-1} + p'$, then we have  $h^0(q_1 + \cdots + q_{g-j-1} + p'-p) = 0$. Thus the general fiber of $\Phi$ is finite.

To conclude the proof, we exhibit a $(g-j)$-dimensional family of degree-$d$ line bundles with basepoints. Let $K = q_1 + \cdots + q_{g-j-1} + D'$ be a canonical divisor with $D'$ effective. Let $p \in C$ be a point different from the $q_i$ and not in the support of $D'$. Define $D = K+p-\sum q_i$. Then we show the following:
\begin{enumerate}[(i)]
	\item $h^1(D) = 0$,
	\item $h^1(D-p) = 1$, i.e., $D$ has a basepoint at $p$.
\end{enumerate}
Indeed $h^1(D) = h^0(K-D'-p)$. But there is a unique canonical section containing $D'$ and it does not contain $p$ by construction. Thus $h^1(D) = 0$. On the other hand, $h^1(D-p) = h^0(K-D') = 1$.

The fiber of $\Phi$ over a divisor $D=K - \sum q_i + p$ constructed in this manner consists of a single point. Indeed, if $D \sim K-\sum q'_i + p'$, then $\sum q_i + p' -p$ has a section. By Serre duality, this is occurs precisely if $h^1(D+p') > 0$. But we have seen that $h^1(D) = 0$, so this is impossible.
\end{proof}

\section{Higher rank bundles}\label{sec:bundles}
In this section we prove the main results about the dimension of the locus $N_C(r,d)$ of stable bundles with basepoints. We work throughout with a smooth projective curve $C$ of genus $g \geq 2$. We begin with a few facts which show that we may restrict our attention to ranks $r$ and degrees $d$ satisfying $rg+1 \leq d \leq r(2g-1)-1$, because this is the range for which the general---but not every---stable bundle is globally generated. The first fact is a well-known result describing the cohomology of a general stable bundle.

\begin{lemma}[\protect{\cite{laumon}\cite{sundaram}\cite{ballico-ramella}}]\label{lemma:gen_bundle_one_nonzero_coh}
Let $C$ be a smooth curve and $E$ a general stable vector bundle on $C$. Then $E$ has at most one nonzero cohomology group.
\end{lemma}

\begin{prop}\label{prop:easy_vb_facts}
Let $C$ be a smooth curve of genus $g \geq 2$. Let $r$ and $d$ be integers with $r \geq 2$. We have the following:
\begin{enumerate}[(a)]
	\item If $d \leq rg$, then the general $E \in U_C(r,d)$ is not globally generated.
    \item If $d > r(g-1
    )$, then $h^1(E)= 0$ for a general $E \in U_C(r,d)$.
	\item If $d > r(2g-1)$, then $E$ is globally generated for all $E \in U_C(r,d)$.
	\item If $d = r(2g-1)$ and $E$ is stable, then $E$ is globally generated.
	\item If $rg +1 \leq d \leq r(2g-1)-1$, then the general $E\in U_C(r,d)$ is globally generated. 
\end{enumerate}
\end{prop}
\begin{proof}
Statement (a) is clear because if $d \leq rg$, then $\chi(E) \leq r$ by Riemann-Roch. By Lemma \ref{lemma:gen_bundle_one_nonzero_coh}, we have $h^0(E) \leq r$ and therefore $E$ does not have enough sections to be globally generated.

For statement (b), we have $h^0(E) - h^1(E) = d + r(1-g) >0$. By Lemma \ref{lemma:gen_bundle_one_nonzero_coh}, it must be that $h^1(E) = 0$

For $E$ to fail to be globally generated, the map $H^0(E) \to E_p$ must fail to be surjective, so we must have $0 < h^1(E(-p)) = h^0( E^\ast(K+p))$. But if $\deg E^\ast(K+p) = -d+r(2g-1) < 0$ and $E^\ast(K+p)$ is semistable, then it cannot have any sections. This gives (c).

Statement (d) is much like statement (c): we have $0 < h^1( E(-p)) = \hom(E, \O(K+p))$. Since both $E$ and $\O(K+p)$ are stable of same slope, there can be no nonzero maps between them unless they are isomorphic, which they are not.

For statement (e), we note that the dimension of the locus of non-globally generated bundles in $U_C(r,d)$ is bounded above by $r^2(g-1) + 1 - (d - rg)$ (see Equation (\ref{eqn:expected_dimension})), which is strictly smaller than the dimension of $U_C(r,d)$ when $ rg+1 \leq d \leq r(2g-1)-1$.
\end{proof}

\subsection{Higher rank bundles with basepoints}
We are now able to compute the dimension of $N_C(r,d)$. The following three facts are the primary objectives of this subsection.
\begin{itemize}
    \item Let $N_C^0(r,d) \subset N_C(r,d)$ denote the set of stable, non--globally generated $E$ with $h^1(E) = 0$. Then every component of $N_C^0(r,d)$ has the expected codimension or consists of vector bundles that are nowhere globally generated, in which case the component has codimension at most $d-rg+1$ (Prop. \ref{prop:determinantal_description}.)
    \item For each $d$ in the range $rg+1 \leq d \leq r(2g-1)-1$, there is a family of non--globally generated vector bundles with the expected dimension (Cor. \ref{cor:family_of_exp_dim}.)
    \item Any stable $E \in U_C(r,d)$ with $rg+1 \leq d \leq r(2g-1)-1$ and a basepoint at $p$ is an extension of the form
    \[
    0 \to F \to E \to \O(K+p-D) \to 0,
    \]
    with $D$ effective (Prop. \ref{prop:upperbound_vb}.)
\end{itemize}

Proposition \ref{thm:existence_stable_seq} constructs an element of $N_C^0(r,d)$. In fact, it produces a family of stable bundles of the form
\[
0 \to F \to E \to \O(K+p-D) \to 0,
\]
where $F$ is a generic stable bundle of rank $r-1$ and $D$ is an effective divisor. The main content of the proof lies in showing that a generic such $E$ is stable. To do so, we follow the method of \cite{teixidor-russo}, which was originally used to show that generic extensions of the form
\[
0 \to E' \to E \to E'' \to 0
\]
are stable whenever $E'$ and $E''$ are generic stable bundles and $\mu(E') < \mu(E'')$. We can use the result of \cite{teixidor-russo} to produce stable bundles with basepoints as follows. Let $D$ be an effective divisor on $C$ whose degree satisfies $2g-1-\deg D > d/r$. Let $L$ be a generic line bundle of degree $2g-1-\deg D$ and $F$ a generic stable bundle of rank $r-1$ and degree $d-(2g-1-\deg D)$. Then \cite{teixidor-russo} says that a generic extension
\[
0 \to F \to E \to L \to 0
\]
is stable. Twisting by $L^\ast(K+p-D)$, we obtain a stable vector bundle $E \otimes L(K+p-D)$ exhibited as an extension
\[
0 \to F\otimes L(K+p-D) \to E\otimes L(K+p-D) \to \O(K+p-D) \to 0.
\]
The bundle $E \otimes L(K+p-D)$ has a basepoint at $p$ as long as $p$ is not in the support of $D$. A slightly stronger statement is available: the dense set from which $F$ is chosen need not depend on $p$ and $D$. The proof of this stronger fact is essentially what appears in \cite[Prop. 1.11]{teixidor-russo} and \cite[Thm. 0.1]{teixidor-russo}, with some extra care taken in tracking the dependence on $p$ and $D$. We include a complete proof below for the reader's convenience.

\begin{prop}\label{thm:existence_stable_seq}Suppose $g \geq 2$, $r \geq 2$, and $rg+1 \leq d \leq r(2g-1)-1$. Let $l$ be a nonnegative integer such that $2g-1-l > d/r$. Then there exists a dense collection of stable vector bundles $F$ of rank $r-1$ and degree $d-(2g-1-l)$ such that for any effective divisor $D$ of degree $l$ and any point $p \in C$ not in the support of $D$, a generic extension  
\[ 0 \to F \to E \to \O(K+p-D) \to 0, \]
yields a stable bundle $E$ of rank $r$ and degree $d$ with a basepoint at $p$.
\end{prop}
\begin{proof} We begin by adding points to the divisor $D$ to make its degree as large as possible. We then use elementary and dual elementary modifications \cite[Def. 1.6]{teixidor-russo} to remove the extra points. Define the integer $j$ by
\[
2g-j- \deg D = \begin{cases}
\lceil \frac{d}{r}\rceil, &  \text{if } r \nmid d\\
\frac{d}{r} + 1,  &\text{if }r \vert d
\end{cases}
\]

By the definition of $j$, we must have $j \geq 1$. If $j \geq 2$, pick generic points $q_1, \dots, q_{j-1} \in C$ with $p \neq q_i$ for all $1 \leq i \leq j-1$. 

We first show that for $F$ generic in its moduli space, a general extension of $\O(K+p-D-q_1 - \cdots - q_{j-1})$ by $F$ is stable. This is the setup of \cite[Prop. 1.11]{teixidor-russo}.  Moreover, we will show that $F$ can be taken to move within an open subset of $U_C(r-1,d-(2g-j-l))$ that does not depend on $p$, $D$, or the points $q_1, \dots, q_{j-1}$. To do this, we retread the argument of \cite[Prop. 1.11]{teixidor-russo} and highlight the genericity assumption on the bundle $F$. We later compute the dimension of this space of extensions and show that it is bounded below by $(r-1)(g-1)+1$ (see Equations (\ref{eqn:extension_space_dimension_1}) and (\ref{eqn:extension_space_dimension_2}).)

Let $E$ be such an extension, so that we have the following exact sequence:
\begin{equation}\label{eqn:exact_seq_lowerdeg}
    0 \to F \to E \to \O(K+p-D-q_1 - \cdots - q_{j-1}) \to 0.
\end{equation}

Dualizing the above sequence, we get an exact sequence exhibiting $E^\ast$ as an extension of $F^\ast$ by $\O(-K-p+D+q_1+\cdots + q_{j-1})$. We have
\begin{equation}\label{eqn:bound_slope_fdual}
\mu(F^\ast) - \mu(E^\ast) = \frac{r(2g-j-\deg D)-d}{r(r-1)} \leq \frac{1}{r-1}.
\end{equation}

Suppose $E^\ast$ is not stable so that there exists a maximal destabilizing subbundle $G \subset E^\ast$ with $\mu(G) \geq \mu(E^\ast)$. Without loss of generality we may assume $G$ is stable. Since $\mu(G) \geq  \mu(E^\ast) > j-2g$, $G$ cannot map to $\O(-K-p+D+q_1 + \cdots + q_{j-1})$. Thus the composition $G \to E^\ast \to F^\ast$ is nonzero. Let $G'$ denote the image of $G$ in $F^\ast$.

We show next that $\rk(G') = r-1$. Suppose not. Then $\rk(G') < r-1$. Since $F$ is generic, we have by \cite[Satz 2.2]{lange},
\begin{equation}\label{eqn:bound_slope_g'}
\mu(G') \leq \mu(F^\ast) - \left(1 - \frac{\rk(G')}{r-1} \right) (g-1).
\end{equation}
Combining (\ref{eqn:bound_slope_fdual}) and (\ref{eqn:bound_slope_g'}), we have
\begin{equation}\label{eqn:bound_slope_all}
\mu(E^\ast) \leq \mu(G) \leq \mu(G') \leq \mu(E^\ast) + \frac{1}{r-1} - \left(1 - \frac{\rk(G')}{r-1}\right)(g-1).
\end{equation}
In particular, we must have
\[
\frac{1}{r-1} - \left(1 - \frac{\rk(G')}{r-1}\right)(g-1) \geq 0.
\]
Equivalently,
\[
(r-1-\rk(G'))(g-1) \leq 1.
\]

If $ g \geq 3$, we have $rk(G') = r-1$ as claimed. Assume now that $ g = 2$. Then we must have $\rk G' = r-2$. The inequality (\ref{eqn:bound_slope_all}) becomes an equality, and consequently the inequalities (\ref{eqn:bound_slope_fdual}) and (\ref{eqn:bound_slope_g'})  become equalities too. Applying (\ref{eqn:bound_slope_fdual}) again, we have
\begin{equation}\label{eqn:dual_slope}
\mu(F^\ast) = \mu(E^\ast) + \frac{1}{r-1}.
\end{equation}

Since $g=2$, by the definition of $j$ we have 
\[ 4 - j - \deg (D) > \frac{d}{r}, \]
and by hypothesis we have 
\[ \frac{d}{r} \geq 2 + \frac{1}{r} . \]
Combining these inequalities with the fact that $D$ is effective, we have 
\[ 2 - j - \frac{1}{r} > \deg (D) \geq 0 , \]
which forces $j=1$ and $D=0$. The above equation (\ref{eqn:dual_slope}) becomes
\[
\frac{3-d}{r-1} = - \frac{d}{r} + \frac{1}{r-1},
\]
and thus $d=2r$, which is not possible because we have assumed $d \geq 2r+1$. We conclude that $\rk G' = r-1$ for all $g \geq 2$.

Now (\ref{eqn:bound_slope_fdual}) and the fact that $\mu(G') \leq \mu(F^\ast)$ together imply the following chain of inequalities.
\begin{equation}\label{eqn:bound_slope_eff}
-\frac{d}{r} \leq \frac{\deg G}{r-1} \leq \frac{-d + 2g-\deg D-j}{r-1} \leq -\frac{d}{r} + \frac{1}{r-1}.
\end{equation}

Multiplying (\ref{eqn:bound_slope_eff}) by $r-1$ we see that one of the following must hold:
\begin{enumerate}[(i)]
    \item $\deg G = -d + 2g-\deg D-j$, or
    \item $\deg G = -d + 2g-\deg D-j-1$.
\end{enumerate}
In case (i), we have $G = F^\ast$ because $F^\ast$ is stable. But $E^\ast$ is a generic extension and the exact sequence (\ref{eqn:exact_seq_lowerdeg}) does not split, so this is impossible. Thus $\deg G = -d + 2g-\deg D-j - 1$ and we have an exact sequence
\begin{equation}\label{eqn:exact_seq_defining_G}
0 \to G \to F^\ast \to \O_{p'} \to 0
\end{equation}
for some $p' \in C$. Consider now the pullback diagram
\[
\xymatrix{
0 \ar[r] & \O(-K-p+D+q_1+\cdots+q_j) \ar[r] & E^\ast \ar[r] & F^\ast \ar[r] & 0\\
0 \ar[r] & \O(-K-p+D+q_1+\cdots+q_j) \ar[r]\ar[u] & E^\ast \times_{F^\ast} G \ar[r]\ar[u] & G \ar[r]\ar[u] & 0
}
\]
The second row splits because $G$ is a subbundle of $E^\ast$, and so the corresponding element of $\Ext^1(F^\ast, \O(-K-p+D+q_1 + \cdots + q_{j-1}))$ is in the kernel of the map
\[
H^1(F(-K-p+D+q_1 + \cdots + q_{j-1})) \to H^1(G^\ast(-K-p+D+q_1+\cdots+q_{j-1})),
\]
and thus is in the image of the map
\[
H^0(\O_{p'}) \to H^1(F(-K-p+D+q_1 + \cdots + q_{j-1})).
\]

Thus given $G$, $E^\ast$ can vary in a family of dimension at most 1. The bundle $G$ itself depends on the choice of point $p'$ in (\ref{eqn:exact_seq_defining_G}) and an element of $\P(F_{p'})$. It follows that the locus of unstable $E^\ast$ has dimension at most $1+(r-2) + 1 = r$ in $\Ext^1(F^\ast, \O(-K-p+D+q_1+\cdots+q_{j-1}))$.

We now show that the space of extensions has dimension greater than $r$. By Serre duality,
\begin{equation}\label{eqn:extension_space_dimension_1}
\ext^1(F^\ast, \O(-K-p+D+q_1+\cdots+q_{j-1})) = h^0(F^\ast(2K+p-D-q_1-\cdots-q_{j-1})).
\end{equation}
Since $d \leq r(2g-j) - r \deg (D) -1$ by the definition of $j$, we must have
\begin{equation}\label{eqn:extension_space_dimension_2}
\begin{aligned}
    h^0(F^\ast(2K+p-D-q_1-\cdots-q_{j-1})) &= \chi(F^\ast(2K+p-D-q_1-\cdots-q_{j-1}))\\
    &= (r-1)(3g-1-\deg D-j)+(-d+2g-\deg D-j)\\
    &\geq (r-1)(g-1)+1.
\end{aligned}
\end{equation}
As a consequence, we have 
\[ 
\ext^1(F^\ast, \O(-K-p+D+q_1+\cdots+q_{j-1})) \geq (r-1)(g-1)+1.
\]
We distinguish two cases: $g = 2$ and $g \geq 3$. When $g \geq 3$, the above gives $h^0(F^\ast(2K+p-D-q_1-\cdots-q_{j-1})) \geq r+1$. We have seen that the space of extensions with $E$ unstable has dimension at most $r$, thus the general one must be stable.

If $g=2$, then $j=1$, $D=0$, and
\[
h^0(F^\ast(2K+p)) = 4r-1-d.
\]
We computed above that $\deg G = 3-3r$. Since $2r + 1 \leq d \leq 3r-1$ by assumption, we see that $h^0(F^\ast(2K+p)) \geq r+1$ unless $d=3r-1$. If $d=3r-1$, then (\ref{eqn:bound_slope_eff}) gives
\[
\frac{1-3r}{r} \leq \frac{\deg G}{r-1} = \frac{3-3r}{r-1},
\]
which is impossible.

Note that $ \rk (G') = r-1$ implies that the map $ G \to G'$ is an isomorphism. Indeed, $\rk (G') = r-1$ forces $\rk(G) \geq r-1$, and since $G$ is a maximal destabilizing subbundle of a non-stable rank-$r$ vector bundle, we have $\rk(G) \leq r-1$. Since the map $G \to G'$ is a surjection between two vector bundles of the same rank, it must be an isomorphism. 

The only requirement on the genericity of $F$ we have used in the above argument is that it needs to satisfy \cite[Satz 2.2]{lange} bounding the slope of its subsheaves. To finish the proof of the theorem, we now use a sequence of elementary and dual elementary modifications to produce a stable $E$ fitting into an exact sequence of the form
\[
0 \to F \to E \to \O(K+p-D) \to 0.
\]

Given an exact sequence of the form (\ref{eqn:exact_seq_lowerdeg}), apply a generic dual elementary modification (\cite[Def. 1.6]{teixidor-russo}) at the point $q_1$. We get an exact sequence (\cite[Lemma 1.8]{teixidor-russo})
\[
0 \to F \to E' \to \O(K+p -D- q_2 - \cdots - q_{j-1}) \to 0.
\]
Applying a generic elementary modification at a general point $p'' \in C$, we get another exact sequence (\cite[Lemma 1.7]{teixidor-russo})
\[
0 \to F' \to E'' \to \O(K+p-D - q_2 - \cdots - q_{j-1}) \to 0.
\]

By varying the point $p''$ we obtain a family of vector bundles with $ \deg (E'') = \deg (E)$. When $p'' = q_1$, there is an elementary modification that recovers the bundle $E$ we began with. Since $E$ is stable, the generic $E''$ in this family is stable. Since $F$ and $p''$ are generic, $F'$ is stable \cite[Lemma 2.5]{ballico-ramella}, and the locus of such $F'$ is dense in the moduli space \cite[Rmk. 2.6]{ballico-ramella}. We can now repeat this process, beginning with an elementary modification at $q_2$ and repeating to eliminate the $q_{i}$'s and produce a stable extension of the desired form. Since the family of bundles $F$ we started with was independent of $p$, $D$, and  the $q_i$, it follows that the family of bundles of rank $r-1$ and degree $d - (2g -1-l)$ we obtain via generic elementary modifications is also independent of $p$ and $D$.
\end{proof}

We now consider the following subsets of $N_C(r,d)$. We will show that they are nonempty open subsets. We will use them while analyzing the irreducible components of $N_C(r,d)$.
\begin{definition}
Let $r \geq 2$ and $d$ be integers. We define the following loci in $U_C(r,d)$.
\begin{align*} N^0_C(r,d) &= \{ E \in N_C(r,d) : h^1(E) = 0 \}\\
N^f_C(r,d) &= \{ E \in N_C(r,d) : h^1(E) = 0 \text{ and $E$ has finitely many basepoints} \} \end{align*}

\end{definition}
Since vanishing of the first cohomology is an open condition, we see that $N^0_C(r,d)$ is an open subset of $N_C(r,d)$ and is nonempty when $d > r(g-1)$ by Proposition \ref{prop:easy_vb_facts}. 

\begin{lemma}\label{lemma:finite_bps}
Suppose $r \geq 1$ and $rg \leq d \leq r(2g-1)-1$ or $r=1$ and $d=2g-1$. Then the locus $N_C^f(r,d)$ is a nonempty open subset of $N_C(r,d)$
\end{lemma}
\begin{proof}
If $r=1$, then $E$ is a line bundle $L$. Since $d \geq g$, we have $h^0(L) > 0$. A section defines a map $\O \to L$ whose cokernel is a torsion sheaf supported on a finite collection of points. The basepoints of $L$ must be a subset of this collection, hence there are finitely many.

Suppose $r \geq 2$. Our first goal is to show that $N^f_C(r,d)$ is an open subset of $N_C(r,d)$. Let $\mathscr{T} \to T \times C$ be a family of stable vector bundles on $C$ with $h^1(\mathscr{T}_t) = 0$ for all $t \in T$. Let $\Sigma$ denote the incidence correspondence
\[
\Sigma = \{(t,p): h^1(\mathscr{T}_t(-p))>0\} \subset T \times C.
\]
Then a bundle $\mathscr{T}_t$ has infinitely many basepoints if and only if the fiber over $t$ is positive-dimensional. Thus the openness of $N_C^f(r,d)$ follows from the semicontinuity of dimension of the fibers.

We proceed to show that $N^f_C(r,d)$ is non-empty. If $r \geq 2$ and $d = rg$, then put $E = L^{\oplus r}$, where $L$ is a degree $g$ line bundle with $h^0(L) = 1$ and $h^1(L) = 0$. Clearly $h^1(E) = 0$. By our discussion for the $r=1$ case, $L$ has finitely many basepoints, and hence, $E$ also has finitely many basepoints. By \cite[Prop. 2.6]{narasimhan-ramanan}, we may deform $E$ to a stable bundle. Let $\mathscr{T} \to T \times C$ be a family containing $E$ and a stable bundle, and let $\pi:T \times C \to T$ be the projection. The map $\pi^\ast \pi_\ast \mathscr{T} \to \mathscr{T}$ restricts to the map $H^0(\mathscr{T}_t) \to (\mathscr{T}_t)_p$ at a point $(t,p) \in T \times C$. When $\mathscr{T}_t = E$ and $p$ is not a basepoint of $E$, this map is surjective, hence it is surjective for a general $(t, p)$. Moreover, $\mathscr{T}_t$ is stable with $h^1(\mathscr{T}_t)=0$ for a general $t$. In particular, $N^f_C(r,rg)$ is nonempty.

Finally, we look at the case $r \geq 2$ and $ rg+1 \leq d \leq r(2g-1)-1$. We induct on the rank and degree. We wish to take an extension of a line bundle by a rank $r-1$ bundle as in Proposition \ref{thm:existence_stable_seq} to apply the inductive hypothesis. Let $j$ be a positive integer satisfying
\[
(r+1)g-d \leq j < 2g - \frac{d}{r}.
\]
Note that $2g - d/r$ is always positive because $d \leq r(2g-1)-1$. Since $j< 2g - d/r$, we have 
\[
\frac{d-2g+j}{r-1} < \frac{d}{r} < 2g-j.
\]
As a consequence of $(r+1)g - d \leq j$ and above inequality, we see that 
\[ (r-1)g \leq d - 2g +j < (r-1)(2g-j)   \] 
Using Proposition \ref{prop:easy_vb_facts} along with the inductive hypothesis, a generic stable bundle $F$ of rank $r-1$ and degree $d-2g+j$ will have $h^1(F) = 0$ and at most finitely many basepoints. By Proposition \ref{thm:existence_stable_seq}, there exists an extension of the form
\[
0 \to F \to E \to \O(K+p-q_1 -\cdots -q_{j-1}) \to 0
\]
with $E$ stable. Since $rg+1 \leq d$, we have $2g - d/r \leq g - 1/r$, and consequently, $j-1 \leq g-1$. If the points $p, q_1, \dots, q_{j-1}$ are chosen to be general, then $h^1(\O(K+p-q_1 -\cdots -q_{j-1})) =0$ because a general effective divisor of degree $j-1$ has at most one global section \cite[IV.1.7]{acgh}. It follows from above exact sequence that $E$ has $p$ as a basepoint. Since $F$ and $O(K+p - q_1 - \cdots - q_{j-1})$ have at most finitely many basepoints and trivial first cohomology, so does $E$. 
\end{proof}

The next proposition shows that the locus $N_C(r,d)$ is determinantal, at least on the open subset of $U_C(r,d)$ consisting of sheaves with the expected number of global sections. To prove this, we use a Poincar\'e family on the moduli space or a suitable \'etale cover.

\begin{prop}\label{prop:determinantal_description}
Let $N_C^0(r,d) \subset N_C(r,d)$ denote the set of stable, non--globally generated $E$ with $h^1(E) = 0$. Then every component of $N_C^0(r,d)$ is either of the expected codimension $d - rg$ or consists of vector bundles that are nowhere globally generated, in which case the component has codimension at most $d-rg+1$.
\end{prop}
\begin{proof}
By \cite[Prop. 2.4]{narasimhan-ramanan}, there exists an \'etale covering $\mathcal{M}\to U_C(r,d)$ that carries a Poincar\'e family $\mathscr{E} \to \mathcal{M} \times C$. When $r$ and $d$ are coprime, we may take $\mathcal{M}$ to be $U_C(r,d)$ itself. The scheme $\mathcal{M}$ parameterizes a family $\{\mathscr{E}_m\}_{m \in \mathcal{M}}$ of stable sheaves of rank $r$ and degree $d$.

Define $Y \subset \mathcal{M}$ to be the open subset $\{m \in \mathcal{M}: h^1(\mathscr{E}_m) = 0\}$. Denote by $\mathscr{E}|_{Y \times C}$ the restriction of the Poincar\'e bundle to $Y \times C$ and let $\pi_1$ and $\pi_2$ denote the projections to the first and second factors, respectively. The sheaf $\pi_1^\ast\pi_{1\ast} \mathscr{E}|_{Y \times C}$ is locally free on $Y$ with fiber over $(m,p)$ equal to $H^0(\mathscr{E}_m)$.

Let $p \in C$ be a point. The natural evaluation map $E \to E_p$ for any vector bundle $E$ induces a map
\[
\pi_1^\ast\pi_{1\ast} \mathscr{E} \to \mathscr{E} \otimes \pi_2^\ast \O_p,
\]
For a point $m \in Y$, write $E=\mathscr{E}_m$ and take fibers of the above map at $(m,p) \in Y \times C$. Then the map becomes
\[
H^0(E) \to E_p,
\]
where the former has rank $\chi(E) = d + r(1-g)$ on $Y$ and the latter has rank $r$. This map fails to be surjective at $p$ if and only if $p$ is a basepoint of $E$. Thus the determinantal locus where the map
\[
\pi_{1\ast} \mathscr{E}|_{Y \times C} \to \pi_{1\ast} (\mathscr{E}|_{Y \times C} \otimes \pi_2^\ast \O_p)
\]
has rank at most $r-1$ is the locus $\Sigma \subset Y \times C$ parameterizing pairs $(E,p)$ where $E$ is stable with a basepoint at $p$. The codimension of this locus is at most $d-rg+1$ by the theory of determinantal varieties \cite[Ch. 2]{acgh}.

As a consequence of Lemma \ref{lemma:finite_bps} the general bundle $E$ has finitely many basepoints, so the fibers of the natural surjective map $\Sigma \to N_C^0(r,d)$ are generically finite. Thus the general point in $N_C^0(r,d)$ is in a component of codimension $d-rg$.

If $E$ is a bundle such that $H^0(E) \to E_p$ is not surjective for all $p$ in $C$, then the fiber of $\Sigma \to N^0_C(r,d)$ is $C$ itself, in which case we see that $E$ is contained in a component of $N^0_C(r,d)$ of codimension at most $d-rg+1$.
\end{proof}

An immediate consequence of Proposition \ref{thm:existence_stable_seq} along with Proposition \ref{prop:determinantal_description} is that $N_C(r,d)$ always contains a component of the expected dimension.

\begin{cor}\label{cor:family_of_exp_dim}
Suppose $rg+1 \leq d \leq r(2g-1)-1$. Then $N_C(r,d)$ has a component of the expected dimension.
\end{cor}
\begin{proof}
Every irreducible component of $N_C(r,d)$ that intersects $N^f_C(r,d)$ non-trivially will have the expected dimension.
\end{proof}

If we impose slightly stronger restrictions on the degree of $E$, we can show that the family of bundles we constructed has the expected codimension, even outside of $N_C^0(r,d)$.

\begin{prop}\label{prop:lowerbound_vb}
Assume $r \geq 2$, $g \geq 2$, $1 \leq j \leq g-1$ and $rg+g-j \leq d \leq r(2g-j)-1$. Then the locus of $E \in U_C(r,d)$ that can be expressed in an exact sequence of the form
\[
0 \to F \to E \to \O(K+p-q_1-q_2-\cdots-q_{j-1}) \to 0
\]
where $p,q_1,q_2, \dots, q_{j-1}$ are in $C$, $p \neq q_i$ for all $i$, and $F \in U_C(r-1,d-(2g-j))$, has the expected codimension $d-rg+(r-1)(j-1)$.
\end{prop}
\begin{proof}
The proof proceeds by counting the dimension of the family parameterizing such extensions, then by studying the classifying map from this family to the moduli space. Proposition \ref{thm:existence_stable_seq} guarantees the existence of at least one stable $E$ of the desired form.

Observe first that there is a universal family parameterizing line bundles of the form $\O(K+p-q_1-\cdots-q_{j-1})$. Indeed, define
\[
\Sigma = \{ (p,q_1, \dots, q_{j-1}) \in C^{\times j} : p \neq q_i \text{ for all } i\}.
\]
Let $\pi: C \times \Sigma \to C$ be the first projection and $\pi_i: C \times \Sigma \to C$ be the composition $C \times \Sigma \to \Sigma \to C$, where the second map is the projection to the $i$th coordinate. Then
\begin{align*}
\pi^\ast\O(K) \otimes (\pi \times \pi_1)^\ast (\O_{C \times C}(\Delta)) \otimes \bigotimes_{i=2}^j (\pi \times \pi_i)^\ast (\O_{C \times C}(-\Delta))|_{C \times (p,q_1,\dots,q_{j-1})}&\\
\cong \O(K+p-q_1 - \cdots -q_{j-1}).&
\end{align*}

By \cite[Prop. 2.4]{narasimhan-ramanan}, if $r \geq 3$, there exists a non-singular, separated scheme $\mathcal{M}$ of finite type parameterizing a family $\mathscr{F} = \{ \mathscr{F}_m \}_{m \in \mathcal{M}}$ of rank $r-1$, degree $d-(2g-1)$ stable vector bundles satisfying the following.
\begin{enumerate}[(i)]
    \item $\mathcal{M}$ has finitely many irreducible components.
    \item $\dim \mathcal{M} = (r-1)^2(g-1) + 1$.
    \item The classifying map $\theta_{\mathscr{F}}: \mathcal{M} \to U_C(r-1, d-(2g-1))$ is \'etale and surjective.
\end{enumerate}

If $r=2$, we write $\mathcal{M} = U_C(1, d-(2g-1))$ and let $\mathscr{F}$ be the universal family of degree $d-(2g-1)$ line bundles. Let $\mathcal{M}_1, \mathcal{M}_2, \dots, \mathcal{M}_n$ be the irreducible components of $\mathcal{M}$. By \cite[Prop. 2.1]{sundaram}, for all $1 \leq i \leq n$ there exists a universal family of extensions $\mathcal{P}_i$ parameterizing the projective spaces $\P(\Ext^1(\O(K+p-q_1-\cdots-q_{j-1}), \mathscr{F}_m))$ as $p$ varies over $C$ and $m$ varies over $\mathcal{M}_i$. Since $\mathscr{F}_m$ is stable and $\mu(\mathscr{F}_m) < 2g-j$, we must have $\hom(\O(K+p-q_1-\cdots-q_{j-1}), \mathscr{F}_m) = 0$. From this we conclude that the fibers $\P(\Ext^1(\O(K+p-q_1-\cdots-q_{j-1}), \mathscr{F}_m)$ are irreducible of the same dimension. Since $\mathcal{M}_i \times C$ is irreducible, so is $\mathcal{P}_i$ for $1 \leq i \leq n$.

Let $\mathcal{P}$ denote the union of the $\mathcal{P}_i$, and let $\mathcal{P}^s \subset \mathcal{P}$ be the open subset consisting of stable extensions. Since $\theta_{\mathscr{F}}$ is \'etale, we have $\dim \mathcal{P}_i = \dim \mathcal{P}_{i'}$ for all $i$, $i'$. This gives,
\[
\dim \mathcal{P}^s = \dim \mathcal{P} = \dim \mathcal{M} + \dim C + \P(\Ext^1(\O(K+p-q_1-\cdots-q_{j-1}), \mathscr{F}_m),
\]
whence
\begin{align*}
\dim \mathcal{P}^s &= [(r-1)^2(g-1)+1] + j + [(r-1)(3g-1-j)+(2g-j)-d-1]\\
&= r^2(g-1) + 1 - (d-rg + (r-1)(j-1)).
\end{align*}

Let $\Phi: \mathcal{P}^s \to U_C(r,d)$ denote the canonical rational map. We show that $\Phi$ is generically finite. For any vector bundle $F$ of degree $d-(2g-j)$ and rank $r-1$, we have
\begin{align*}
-\chi(F) &= \chi(F^\ast(K)) = rg+g-r+1-j-d < 0,\\
\deg(F) &= d-(2g-j) \geq (r-1)g.
\end{align*}

From Lemma \ref{lemma:gen_bundle_one_nonzero_coh}, we see that the general $F \in U_C(r-1, d-(2g-1))$ has $h^1(F)=h^0(F^\ast(K)) = 0$. Since $\theta_{\mathscr{F}}$ is \'etale and hence open, the general $F$ in $\theta_{\mathscr{F}}(\mathcal{M}_i)$ must satisfy $h^0(F^\ast(K)) = 0$. Let $V \subset \mathcal{P}^s$ be the open subset parameterizing stable  extensions 
\[ 0 \to F \to E \to \mathcal{O}(K + p - q_1 - \cdots - q_{j-1} ) \to 0 \]
such that $F$ is stable and $h^1(F) = h^1(F(-p)) = 0$.

We claim that $\Phi|_V$ has finite fibers. Let $E$ be any vector bundle in the image $\Phi(V)$. Then there is an extension
\[
0 \to F \to E \to \O(K+p-q_1-\cdots-q_{j-1}) \to 0
\]
with $F$ stable and $h^1(F) = h^1(F(-p)) = 0$. Twisting the above exact sequence by $\O(-p+q_1+\cdots+q_{j-1})$ and taking the long exact sequence in cohomology, we see that $h^1(E(-p + q_1 + \cdots + q_{j-1} )) = 1$. Then $\hom (E, \O(K+p - q_1 - \cdots - q_{j-1})) = 1$ by Serre duality, and so there is a unique map $E \to \O(K+p - q_1 - \cdots - q_{j-1})$ modulo scalars. This map uniquely determines $F$. Moreover, any other point in $\Phi^{-1}(E)$ must be an extension of the form
\[
0 \to F \to E \to \O(K+p'-q_{1}'-\cdots-q_{j-1}') \to 0,
\]
where $p'$ is some (potentially different) basepoint of $E$. Since the general $E$ has finitely many basepoints, we conclude that
\[
\dim \Phi(V) = \dim V = \dim \mathcal{P}^s = r^2(g-1) + 1 - (d-rg+(r-1)(j-1)).
\]
\end{proof}

The next proposition shows that the expected dimension is an upper bound on the dimension of a component of $N_C(r,d)$. Moreover, it shows that any $E$ in $N_C(r,d)$ must be an extension similar to one of the ones constructed in Proposition \ref{thm:existence_stable_seq} because any $E$ with a basepoint at $p$ must admit a map to $\O(K+p)$.

\begin{prop}\label{prop:upperbound_vb}
Any stable $E \in U_C(r,d)$ with $rg+1 \leq d \leq r(2g-1)-1$ and a basepoint at $p$ is an extension of the form
    \[
    0 \to F \to E \to \O(K+p-D)\to 0,
    \]
    with $D$ effective and $p$ not in the support of $D$. Furthermore, for any integer $j \geq 1$, the dimension of the family of such extensions, as $p$ varies over $C$, $D$ varies over the symmetric product $C^{(j-1)}$, and $F$ varies over deformation families of rank $r-1$, degree $d-(2g-j)$ vector bundles whose general element is stable, is at most
    \[
    r^2(g-1) + 1 - (d-rg+(r-1)(j-1)).
    \]
\end{prop}
\begin{proof}
Let $E$ be a stable bundle with a basepoint at $p \in C$. Then $h^1(E(-p)) > h^1(E)$, and so $\hom(E, \O(K+p)) > \hom(E, \O(K))$. Let $f:E \to \O(K+p)$ be a nonzero map whose image $L \subset \O(K+p)$ is not contained in $\O(K)$, and let $F$ be its kernel, so that we have an exact sequence
\begin{equation}\label{eqn:exactseq_bundlewithbp}
0 \to F \to E \to L \to 0.
\end{equation}
Then $F$ is a vector bundle of rank $r-1$ and degree $d - \deg L$. Since the moduli space of semistable sheaves is dense in the stack of coherent sheaves, $F$ can be deformed to a stable bundle and can therefore depend on at most $(r-1)^2(g-1)+1$ moduli (see \cite[Prop 2.6]{narasimhan-ramanan}).

Since we have assumed $L$ is not a subbundle of $\O(K)$, we know that $0 = \hom(L, \O(K)) = h^1( L)$. Observe that $L$ has a basepoint at $p$: by construction, $h^1( L(-p)) = \hom(L, \O(K+p))$ is nonzero, whereas $h^1(L) = \hom(L, \O(K)) = 0$. Thus $p$ is a basepoint of $L$ and by Lemma \ref{lemma:linebundles_with_bp}, $L$ must be of the form $L=\O(K+p-q_1-\cdots-q_{j-1})$ for some points $q_1, \dots, q_{j-1}$. The dimension of the locus of such line bundles, as $p$ and the $q_i$ vary, is $j$. Furthermore, the dimension of the space of such extensions is given by
\[
\ext^1(\O(K+p-q_1-\cdots-q_{j-1}), F) = h^0(F^\ast\otimes L(K)).
\]
To compute this, we dualize the exact sequence (\ref{eqn:exactseq_bundlewithbp}) and twist by $L(K)$:
\[
0 \to \O(K) \to E^\ast \otimes L(K) \to F^\ast\otimes L(K) \to 0.
\]
We show that $h^1(F^\ast\otimes L(K)) = \hom(L,F) = 0$. If there were a nontrivial morphism $L \to F$, then the composition
\[
E \to L \to F \to E
\]
gives a nontrivial morphism $E \to E$ that is not a homothety. This is impossible because $E$ is stable. Thus $h^1(F^\ast \otimes L(K)) =0$, and so
\[
h^0(F^\ast \otimes L(K)) = 3(r-1)(g-1) + (r-1)-r(j-1)-d+2g-1
\]
The bundle $F$ can depend only on at most $(r-1)^2(g-1)+1$ moduli. Thus by moving $F$, the points $p, q_1, \dots, q_{j-1}$, and the extension in $\Ext^1(L, F)$, we obtain a family of dimension at most
\begin{align*}
    [(r-1)^2(g-1) + 1] + j + [3(r-1)(g-1) + (r-1)-r(j-1)-d+2g-1]-1&\\
    = r^2(g-1) + 1 - (d-rg+(r-1)(j-1))&
\end{align*}

Moreover, there are finitely many such families: they are determined by the degree of the line bundle $L$ in (\ref{eqn:exactseq_bundlewithbp}), and we have $\mu(E) \leq \deg L \leq 2g-1$. Thus any semistable bundle with a basepoint is in one of finitely many families of codimension at least $d-rg$.
\end{proof}

\subsection{Irreducibility of $N_C(r,d)$}
We explain here that the locus $N_C(r,d)$ is irreducible of the expected codimension $d-rg$ when $rg +g -1 \leq d \leq r(2g-1)-1$. The key idea is that $U_C(r-1,d-(2g-1))$ is irreducible, so there exists a family of stable bundles parameterized by an irreducible variety such that the classifying map surjects onto $U_C(r-1,d-(2g-1))$. We use this fact together with the constructions in the previous subsection to produce an irreducible subset of $N_C^0(r,d)$ and show that it is dense in $N_C(r,d)$ (see Theorem \ref{thm:irreducibility_n0}). We will use the same notation as in the previous subsection and continue to assume that $C$ is a smooth curve of genus $g \geq 2$.

\begin{lemma}\label{lemma:gen_elt_nc0}
Suppose $r \geq 2$ and $rg+g-1 \leq d \leq r(2g-1)-1$. Then the general element $E$ of any irreducible component of $N_C^0(r,d)$ is an extension of the form
\[
0 \to F \to E \to \O(K+p) \to 0.
\]
\end{lemma}
\begin{proof}
By Proposition \ref{prop:determinantal_description}, every component of $N_C^0(r,d)$ has dimension
\[
r^2(g-1)+1 - (d-rg + \delta),
\]
where $\delta \in \{0,1\}$. On the other hand, Proposition \ref{prop:lowerbound_vb} shows that any $E$ in $N_C(r,d)$ is of the form
\[
0 \to F \to E \to \O(K+p-q_1 - \cdots -q_{j-1}) \to 0
\]
and such families have dimension at most $r^2(g-1)+1 - (d-rg+(r-1)(j-1))$. We see that if $r \geq 3$, then the general $E$ must fit into such an exact sequence with $j=1$ and the claim is proved.

If $r=2$, the general $E$ must fit into such an exact sequence with $j=1$ or $j=2$, and $F$ must be a line bundle. If $j=1$, we are done. If $j=2$, then we have
\[
h^0(F) \geq \chi(F) = (d-(2g-2))+1-g = 2.
\]
Since $F$ has a nonzero global section, it has at most finitely many basepoints and hence $E$ does too. But Proposition \ref{prop:determinantal_description} says that if $\delta = 1$, then the general element of the component of $E$ must have the entire curve as its base locus, a contradiction.
\end{proof}

The next lemma shows that the sets $N_C^0(r,d)$ and $N_C^f(r,d)$ are dense in $N_C(r,d)$. Recall that $N_C^f(r,d)$ is the locus of stable vector bundles $E \in N_C(r,d)$ with finitely many basepoints and $h^1(E) = 0$.

\begin{prop}\label{lemma:dense_openness_nc0_ncf}
\begin{enumerate}[(a)]
    \item Suppose $r \geq 2$ and $\max\{rg+1,rg+g-r+1\} \leq d \leq r(2g-1)-1$. Then $N_C^0(r,d)$ is a dense open subset of $N_C(r,d)$.
    \item Suppose one of the following hold:
    \begin{enumerate}[(i)]
        \item $r \geq 2$ and $rg+g-1 \leq d \leq r(2g-1)-1$, or
        \item $r=1$ and $g \leq d \leq 2g-1$, or
        \item $r \geq 1$ and $d=rg$.
    \end{enumerate}
    Then $N_C^f(r,d)$ is a dense open subset of $N_C(r,d)$.
\end{enumerate}
\end{prop}
\begin{proof}
We have seen that both $N_C^0(r,d)$ and $N_C^f(r,d)$ are nonempty open subsets of $N_C(r,d)$. It remains to show denseness under the hypotheses of the lemma.

To show (a), suppose $E \in N_C(r,d)$ has a basepoint at $p$. Then by Proposition \ref{prop:lowerbound_vb}, there is an exact sequence
    \[
    0 \to F \to E \to \O(K+p - q_1 -\cdots -q_{j-1}) \to 0
    \]
    for some points $q_1, \dots, q_{j-1}$. Note that
    \begin{align*}
        \chi(F) &= d-(2g-j) + (r-1)(1-g)\\
        &\geq (rg+g-r+1)-(2g-j)+(r-1)-(rg-g)\\
        &\geq 1.
    \end{align*}
We will construct a subset of $N_C^0(r,d)$ whose closure contains $E$. Let $\mathscr{F} \to T \times C$ be a family of vector bundles of rank $r-1$ and degree $d-(2g-j)$ parameterized by an irreducible variety $T$ containing $\mathscr{F}_0=F$ and a stable bundle $\mathscr{F}_1$ satisfying $h^1(\mathscr{F}_1) = 0$ (see \cite[Prop 2.6]{narasimhan-ramanan}). Define the incidence correspondence $\Sigma \subset T \times C^{\times j}$ by
\[
    \Sigma = \{(t,p',q_1',\dots,q_{j-1}') : p' \neq q'_{i} \text{ for all }i, \hom(\O(K+p'-q_1'-\cdots-q_{j-1}',\mathscr{F}_t) = 0\}.
    \]
Note that $\Sigma$ is nonempty because $(0,p,q_1,\dots,q_{j-1}) \in \Sigma$. Let $p:\mathcal{P} \to \Sigma$ denote the family parameterizing extensions $\P(\Ext^1(\O(K+p'-q_1'-\cdots-q_{j-1}'), \mathscr{F}_t))$ as $(t,p',q_1',\dots,q_{j-1}')$ varies over $\Sigma$. By the definition of $\Sigma$ and the fact that $\Sigma$ is irreducible, we see that the fibers $\Ext^1(\O(K+p'-q_1'-\cdots-q_{j-1}'), \mathscr{F}_t)$ of $\mathcal{P}$ are irreducible and equidimensional. Thus $\mathcal{P}$ itself is irreducible.

Define the open subset $\Sigma^\circ \subset \Sigma$ by
\[
\Sigma^\circ = \{(t,p',q_1',\dots,q_{j-1}') : h^1(\mathscr{F}_t) = h^1(\O(K+p'-q_1'-\cdots-q_{j-1}') = 0\}.
\]
Observe that for general $p', q_1',\dots, q_{j-1}'$, we have $\hom(\O(K+p'-q_1'-\cdots-q_{j-1}',\mathscr{F}_1) = 0$ because $\mathscr{F}_1$ is stable of slope $(d-(2g-j))/(r-1) < 2g-j$. Thus $\Sigma^\circ$ is nonempty and clearly open in $\Sigma$.

Let $\mathcal{P}^s \subset \mathcal{P}$ denote the subset of $\mathcal{P}$ parameterizing stable extensions. Then we have the canonical rational map
\[
\Phi: \mathcal{P}^s \to U_C(r,d).
\]
Note that $\Phi(\mathcal{P}^s) \cap N_C^0(r,d)$ is nonempty. Indeed, $\mathcal{P}$ is irreducible and both $\mathcal{P}^s$ and $p^{-1}(\Sigma^\circ)$ are nonempty open subsets. Thus $\mathcal{P}^s \cap p^{-1}(\Sigma^\circ)$ is nonempty, and any point in the intersection maps to $N_C^0(r,d)$ under $\Phi$. That is, $\Phi(\mathcal{P}^s \cap p^{-1}(\Sigma^\circ))$ is contained in $N_C^0(r,d)$ and $E$ is in its closure, thus $E$ is in the closure of $N_C^0(r,d)$. This gives (a).

We now prove the statements in part (b).

\begin{enumerate}[(i)]
    \item We argue as in (a) and omit most of the details. Note that the assumption $d \geq rg+g-1$ implies $\chi(F) \geq r-1$, which allows us to deform the bundle $F$ to one that is not only stable but also has finitely many basepoints. In the definition of the incidence correspondence $\Sigma^\circ$, we require further that $\mathscr{F}_t$ have finitely many basepoints.
    \item Since $g \leq d \leq 2g-1$, every line bundle of degree $d$ has a non-zero global section. The zero locus of this section is a finite collection of points containing the base locus of $E$.
    \item The assumption $d=rg$ implies that the general $E$ in $U_C(r,d)$ has $h^0(E) = r$ and $h^1(E) = 0$. In particular, the general $E$ in $U_C(r,d)$ has finitely many basepoints and so $N_C^f(r,d)$ is dense in $U_C(r,d)$ and \emph{a fortiori} is also dense in $N_C(r,d)$.
\end{enumerate}
\end{proof}

\begin{thm}\label{thm:irreducibility_n0}
Suppose $rg+g-1 \leq d \leq r(2g-1)-1$. Then $N_C(r,d)$ is irreducible of the expected codimension $d-rg$.
\end{thm}
\begin{proof}
We observe that by Proposition \ref{lemma:dense_openness_nc0_ncf} it suffices to show that $N_C^0(r,d)$ is irreducible. Using, for example, \cite[Prop 2.6]{narasimhan-ramanan}, we see that there exists an irreducible smooth variety $\mathcal{M}$ together with a family $ \mathscr{F} \to \mathcal{M} \times C$ of stable vector bundles of rank $r$ and degree $d$ such that the classifying map 
\[ \theta_{\mathscr{F}} : \mathcal{M} \to U_C(r-1,d - (2g-1)) \] 
is surjective. Let $\mathcal{P}$ denote the family parameterizing extensions of the form
\begin{equation}\label{eqn:general_exact_seq}
0 \to \mathscr{F}_m \to E \to \O(K+p) \to 0
\end{equation}
as $m$ varies over $\mathcal{M}$ and $p$ varies over $C$. Then $\mathcal{P}$ is irreducible, as is the open subset $\mathcal{P}^s \subset \mathcal{P}$ parameterizing stable bundles $E$. Consequently, $\Phi(\mathcal{P}^s) \cap N_C^0(r,d)$ is irreducible in $N_C(r,d)$. As a consequence of Lemma \ref{lemma:gen_elt_nc0}, the general element $E$ in $N_C^0(r,d)$ can be expressed as 
\begin{equation}\label{eqn:general_elt_N_0}
0 \to F \to E \to \mathcal{O}(K+p) \to 0.
\end{equation}
We wish to show that such an $E$ is in the closure of $\Phi(\mathcal{P}^s)\cap N_C^0(r,d)$. Note that the bundle $F$ may not be stable, but there is a smooth irreducible variety $T$ parameterizing a family of vector bundles $\mathscr{T} \to T \times C$ with fiber $\mathscr{T}_0 = F$ and $\mathscr{T}_t$ stable for the general $t \in T$ (see \cite[Prop 2.6]{narasimhan-ramanan}). Define the incidence correspondence
\[
\Sigma = \{(t,p') : \hom(\O(K+p'), \mathscr{T}_t) = 0 \} \subset T \times C.
\]
Note that $\Sigma$ is open and nonempty because $(0,p) \in \Sigma$. Let $\mathcal{B}$ denote the space parameterizing the projective spaces $\P(\Ext^1(\O(K+p'),\mathscr{T}_t))$. Since $\Sigma$ is irreducible and the fibers of $\mathcal{B} \to \Sigma$ are irreducible of the same dimension, it follows that $\mathcal{B}$ is irreducible. 

Let $\mathcal{B}^s$ be the nonempty open subset of $\mathcal{B}$ parameterizing stable vector bundles, and let $\Psi: \mathcal{B}^s \to U_C(r,d)$ denote the canonical rational map. The general element $(t,p') \in \Sigma$ has $h^1(\mathscr{T}_t) = 0$, $\mathscr{T}_t$ stable, and $\mathscr{T}_t = \mathscr{F}_m$ for some $m \in \mathcal{M}$. Thus the general element of $\Psi(\mathcal{B}^s)$ comes from an exact sequence
\[
0 \to \mathscr{T}_t \to E' \to \O(K+p') \to 0
\]
with $h^1(\mathscr{T}_t) = 0$. Since $h^1(\O(K+p')) =0$, we must have $h^1(E') = 0$. Thus $E'$ is in $\Phi(\mathcal{P}^s) \cap N_C^0(r,d)$. In particular, $E$ is in the closure of $\Phi(\mathcal{P}^s ) \cap N_C^0(r,d)$, which we have seen is irreducible. Since the general element of any irreducible component is of the form (\ref{eqn:general_elt_N_0}) by Lemma \ref{lemma:gen_elt_nc0}, it follows that $N_C^0(r,d)$ is itself irreducible.
\end{proof}

\bibliographystyle{plain}

\end{document}